\makeatletter

\documentclass[12pt]%
{article}

\newsavebox{\savepar}

\usepackage{amsmath}
\usepackage{amsfonts}
\usepackage{euscript}
\usepackage{oldgerm}
\usepackage{amsthm}
\usepackage{rotating}
\usepackage{mathrsfs}
\usepackage{hyperref}
\usepackage{amssymb}
\usepackage{epsfig}
\usepackage{dsfont}
\usepackage{subfig}
\makeindex

\newtheorem{corollary}{Corollary}

  \newtheorem{theorem}{Theorem}[section]

  \theoremstyle{definition}

  \theoremstyle{remark}
  \newtheorem{remark}[theorem]{Remark}

\theoremstyle{definition}

\theoremstyle{remark}

\begin{document}

\newcommand{\norm}[1]{\left\lVert #1\right\rVert}
\newcommand{\namelistlabel}[1]{\mbox{#1}\hfil}
\newenvironment{namelist}[1]{%
\begin{list}{}
{
\let\makelabel\namelistlabel
\settowidth{\labelwidth}{#1}
\setlength{\leftmargin}{1.1\labelwidth}
}
}{%
\end{list}}

\newcommand{\inp}[2]{\langle {#1} ,\,{#2} \rangle}
\newcommand{\vspan}[1]{{{\rm\,span}\{ #1 \}}}
\newcommand{\R} {{\mathbb{R}}}

\newcommand{\B} {{\mathbb{B}}}
\newcommand{\C} {{\mathbb{C}}}
\newcommand{\N} {{\mathbb{N}}}
\newcommand{\Q} {{\mathbb{Q}}}
\newcommand{\LL} {{\mathbb{L}}}
\newcommand{\Z} {{\mathbb{Z}}}

\newcommand{\BB} {{\mathcal{B}}}

\title{A note on the 2-dual space of $L^p[0,1]$  }
\author{ Akshay S. RANE \footnote{Department of Mathematics, Institute of Chemical Technology, Nathalal Parekh Marg, Matunga, Mumbai 400 019, India, email :  as.rane@ictmumbai.edu.in,} 
\hspace {1mm}
}
\date{ }
\maketitle
\begin{abstract}
  	In the present note, we are interested in bounded 2-functionals and 2-dual spaces of $L^p[0,1]$. The 2-dual spaces of the sequence space $l^p$ is considered in the literature. But interestingly an explicit computation of $Lp$ spaces has not been considered though n-duals of general normed spaces have been considered. We shall consider the 2-dual spaces with the usual $\|.\|_p$ norm and with respect to the G\"{a}hler and the Gunawan norm. The n-dual space of $L^p[0,1]$ can be treated in a similar manner.
\end{abstract}



\newpage
\section{INTRODUCTION}
Let $n$ be a non-negative integer and $X$ be a vector space over $\mathbb{R}$ of dimension $d\geq n.$ An n-norm is a real valued function on $X^n$ satisfying the following properties.
\begin{enumerate}
	\item $\|x_1,\ldots,x_n\|=0$ if and only if $x_1,x_2,\cdots,x_n$ are linearly dependent.
	\item $\|x_1,\ldots,x_n\|$ is invariant under permutation.
	\item $\|\alpha x_1,\ldots,x_n\|=|\alpha| \|x_1,\ldots,x_n\|$ for $\alpha \in \mathbb{R}$.
	\item $\|x_1+x_1^\prime ,\ldots,x_n\|\leq \|x_1 ,\ldots,x_n\|+\|x_1^\prime ,\ldots,x_n\|$
\end{enumerate}
Then  $(X,\|.,.\|)$ is called as an n-normed space. Details about n-normed spaces can be found in  \cite{Chen}, \cite{Gun1}, \cite{Gun2}, \cite{Gun3}, and  \cite{Gun4}.
Let $1 \leq p < \infty$ and $q$ be the conjugate exponent of $p$ that is $\frac{1}{p}+ \frac{1}{q}=1.$ The sequence spaces $l^p$ have been studied in the literature. See \cite{Yos1}, \cite{Yos2} and \cite{Gun5}. Analogously
one can equip n-norm on the space $L^p[0,1]$, the space of all bounded measurable real valued functions that are p-th power Lebesgue integrable.That is 
$$L^p[0,1]:\left \lbrace x : [0,1] \rightarrow \mathbb{R} : \int_0^1 |x(t)|^p dt < \infty \right \rbrace.$$   We consider the 2-norm. An n-norm can be similarly defined.  
The following 2-norm on the space $L^p[0,1]$ is due to G\"{a}hler. ( See \cite{Gah1}, \cite{Gah2} and \cite{Gah3}.)
For $x_1,x_2 \in L^p[0,1],$
$$ \|x_1,x_2\|_p^G := \sup_{y_1,y_2 \in L^q[0,1],\|y_1\|_q\leq 1,\|y_2\|_q\leq 1} \begin{vmatrix} \int_0^1 x_1(u) y_1(u) du & \int_0^1 x_2(u) y_1(u) du  \\ \int_0^1 x_1(u) y_2(u) du  & \int_0^1 x_2(u) y_2(u) du  \end{vmatrix} $$ Here $L^q[0,1]$ is the usual dual of
the space $L^p[0,1].$
\noindent
The 2-norm introduced by Gunawan \cite{Gun6} and \cite{Sak} on the space $L^p[0,1]$ takes the following form
 
$$ \|x_1,x_2\|_p^H :=\left ( \frac{1}{2!}\int_0^1 \int_0^1\begin{vmatrix} x_1(u) & x_1(v) \\ x_2(u) & x_2(v) \end{vmatrix}^p  du dv\right )^\frac{1}{p}.$$ \noindent
\noindent
Any real valued function $f$ on $X^n$, is called n-functional on $X.$ Further if $f$ satisfies
\begin{enumerate}
	\item $f(x_1+y_1,\ldots,x_n+y_n) =\displaystyle \sum_{h_i \in { x_i,y_i },1 \leq i \leq n} f(h_1,\ldots,h_n),$
	\item $f(\alpha_1 x_1,\ldots,\alpha_n x_n)=\alpha_1\ldots\alpha_nf(x_1,\ldots,x_n), $

\end{enumerate}
	is called multilinear n-functional on $X.$\\
	An n-functional on a normed space $X$ is said to be bounded on $X$ if there is a constant $K>0$ such that
	$$ |f(x_1,x_2,\ldots,x_n)| \leq K \|x_1\|\ldots\|x_n\| $$
\noindent	
	If $X$ is equipped with the n-norm then n-functional $f$ is said to be bounded if 
	$$ |f(x_1,\ldots,x_n)| \leq K \|x_1,\ldots,x_n\|$$
	
	\begin{remark}
		Every bounded multilinear n-functional $f$ on an n-normed space $X$ satisfies
		$$f(x_1,\ldots,x_n)=0,$$ whenever $x_1,\ldots,x_n$ are linearly dependent. Also $f$ is antisymmetric, that is $$f(x_1,\ldots,x_n)= sgn(\sigma) f(x_{\sigma(1)},\ldots, x_{\sigma(n)})$$ for any $x_1,x_2,\ldots,x_n \in X$ and for any permutation $\sigma$ of $\left \lbrace1,2,\ldots,n\right \rbrace$
	\end{remark} \noindent
The space of bounded multilinear n-functionals on $(X,\|.\|)$ is called the n dual space of $(X,\|.\|)$. The norm on this space is given by
$$ \|f\|_{n,1} := \sup_{\|x\|_1,\ldots,\|x\|_n\neq 0} \frac{|f(x_1,\ldots,x_n)|}{\|x_1\|\ldots \|x_n\|}$$\noindent
whereas 
\noindent
The space of bounded multilinear n-functionals on $(X,\|.,\ldots,.\|)$ is called the $n-$dual space of $(X,\|.\|)$. The norm on this space is given by

$$ \|f\|_{n,n} := \sup_{\|x_1,\ldots,x_n\|\neq 0} \frac{|f(x_1,\ldots,x_n)|}{\|x_1,\ldots,x_n\|}$$
The n-dual spaces of the sequence space was considered by Pangalela and Gunawan \cite{Yos1}.
He considered and identified the n-dual spaces with respect to G\"{a}hler norm and the Gunawan norm. He worked on the more general normed spaces and studied the n-dual in \cite{Yos3}. But an explict computation and identification of the 2-dual space of $L^p[0,1]$ with the natural $\|.\|_p$ norm, G\"{a}hler norm and Guanawan norm were not considered. 
In this paper, we shall consider the 2-dual space of $L^p[0,1].$ The n-dual space computation can be done on similar lines.

\section{THE 2-DUAL SPACES OF $L^p[0,1]$}

We shall here identify the 2-dual space of $L^p[0,1]$ as a normed space. We shall assume that $1\leq p < \infty$ and $q$ is the conjugate exponent of $p$. In other words $q$ satisfies
$\frac{1}{p}+ \frac{1}{q} =1.$
We introduce the following normed space  $Y_{[0,1] \times [0,1]}^q$ as follows :\\
$ Y_{[0,1] \times [0,1]}^q : = $
$$\left \lbrace \theta:[0,1] \times [0,1] \rightarrow \mathbb{R}, \text{measurable} : \|\theta\|_{Y_{[0,1] \times [0,1]}^q}:= \sup_{\|x\|_p=1} 
(\int_0^1|\int_0^1 x(u) \theta(u,v)du|^q dv )^\frac{1}{q} < \infty \right \rbrace.$$
\noindent
$\|\theta\|_{Y_{[0,1] \times [0,1]}^q}$ defines a norm on $Y_{[0,1] \times [0,1]}^q.$
\noindent
In a similar manner, one defines the space $ Y_{[0,1] \times [0,1]}^\infty$
$$\left \lbrace \theta:[0,1] \times [0,1] \rightarrow \mathbb{R}, \text{measurable,\; essentially\; bounded} : \|\theta\|_{Y_{[0,1] \times [0,1]}^\infty} < \infty \right \rbrace,$$
where $$ \|\theta\|_{Y_{[0,1] \times [0,1]}^\infty}:= \sup_{\|x\|_p=1} 
(\sup_{u\in[0,1]}|\int_0^1 x(u) \theta(u,v)du| ) < \infty. $$
We first prove the main result.
\begin{theorem}
 Let $1< p < \infty $, then the 2-dual space of $(L^p[0,1],\|.\|_p)$ is isometrically bijective to $(Y_{[0,1] \times [0,1]}^q,\|.\|_{Y_{[0,1] \times [0,1]}^q})$ 
\end{theorem}
\begin{proof}
Let $ \theta \in Y_{[0,1] \times [0,1]}^q.$ Define a 2-functional on $L^p[0,1] \times L^p[0,1]$ given by

$$ f(x,y) := \int_0^1 \int_0^1 x(u) y(v) \theta (u,v) du dv,$$
where $x,y \in L^p[0,1],$ with $\|x\|_p=\|y\|_p=1.$
On applying H\"{o}lder's inequality, we get
\begin{align*}
|f(x,y)| = | \int_0^1  y(v) ( \int_0^1 x(u) \theta (u,v) du)dv|\leq (\int_0^1 |y(v)|^p dv)^\frac{1}{p} (\int_0^1|\int_0^1 x(u) \theta(u,v)du|^q dv )^\frac{1}{q} 
\end{align*}
As a consequence
$$ |f(x,y)|\leq (\int_0^1|\int_0^1 x(u) \theta(u,v)du|^q dv )^\frac{1}{q}\leq \sup_{\|z\|_p=1}(\int_0^1|\int_0^1 z(u) \theta(u,v)du|^q dv )^\frac{1}{q}=\|\theta\|_{Y_{[0,1] \times [0,1]}^q} $$ 
Hence for $x,y \neq 0,$
$$\frac{|f(x,y)|}{\|x\|_p \|y\|_q}\leq \|\theta\|_{Y_{[0,1] \times [0,1]}^q}$$
As a result \begin{equation}\|f\|_{2,1}\leq  \|\theta\|_{Y_{[0,1] \times [0,1]}^q}\end{equation}
Conversely,  $f$ be a bounded 2-functional on $L^p[0,1]$. That is $f: L^p[0,1] \times L^p[0,1] \rightarrow \mathbb{R}$ is a bounded linear functional. Morever
$$|f(x,y)| \leq \|f\|_{2,1} \|x\|_p \|y\|_p.$$
Let $x \in L^P[0,1]$ with $\|x\|_p=1. $ 
Let $$f_{x} : L^p[0,1] \rightarrow \mathbb{R}$$ be a functional given by
$$ f_{x}(y):= f(x,y).$$ Then 
$$\frac {|f_{x}(y)|} { \|y\|_p}=\frac {|f(x,y)|} { \|y\|_p} \leq \|f\|_{2,1} \|x\|_p.$$
$f_{x}$ is a bounded linear functional on $L^p[0,1]$ and
$$\|f_x\| \leq \|f\|_{2,1}.$$ $f_{x} \in L^p[0,1]^\prime.$
By the Riesz representation theorem, there exist $z(x) \in L^q[0,1] $ such that 
$$ f(x,y)=f_{x}(y) = \int_0^1 z(x)(v) y(v) dv. $$
and $$\|f_{x}\|= \|z(x)\|_q= \left (\int_0^1 |z(x)(v)|^q dv \right )^\frac{1}{q}] \leq \|f\|_{2,1}.$$
For each $x \in L^p[0,1]$ $$ \left (\int_0^1  |z(x)(v)|^q dv\right )^\frac{1}{q} \leq \|f\|_{2,1}.$$
$x \rightarrow z(x)$ is a Bounded linear map from $L^p[0,1]$ to $L^q[0,1].$
Set $$\theta(u,v)=z(x(u))(v).$$
$$ |\int_0^1 x(u) \theta(u,v) du | \leq \left(\int_0^1 |x(u)|^p du \right )^\frac{1}{p}  \left (\int_0^1  |\theta(u,v)|^q du\right )^\frac{1}{q}.$$
As a result
$$ |\int_0^1 x(u) \theta(u,v) du |^q \leq \|x\|_p^q \int_0^1  |\theta(u,v)|^q du. $$
So
$$\int_0^1 |\int_0^1 x(u) \theta(u,v) du |^q dv  \leq \|x\|_p^q \int_0^1 \int_0^1  |\theta(u,v)|^q du\; dv = \|x\|_p^q \int_0^1 \int_0^1  |\theta(u,v)|^q dv\; du $$
$$\int_0^1 |\int_0^1 x(u) \theta(u,v) du |^q dv  \leq \|x\|_p^q \int_0^1  \|f\|_{2,1}^q \; du $$ and

$$\left (\int_0^1 |\int_0^1 x(u) \theta(u,v) du |^q dv \right )^\frac{1}{q} \leq \|x\|_p  \|f\|_{2,1} $$

\begin{equation}\|\theta\|_{Y_{[0,1]\times[0,1]}^q}=\sup_{\|x\|_p=1}\left (\int_0^1 |\int_0^1 x(u) \theta(u,v) du |^q dv \right )^\frac{1}{q} \leq  \|f\|_{2,1}\end{equation}
Thus $$ \|f\|_{2,1}= \|\theta\|_{Y_{[0,1]\times[0,1]}^q}$$ from $(2.1)$ and $(2.2).$\\
Thus $f\rightarrow \theta$ is an isometry from the 2-dual space of $(L^p[0,1],\|.\|_p)$ to $(Y_{[0,1] \times [0,1]}^q,\|.\|_{Y_{[0,1] \times [0,1]}^q})$
\end{proof}
\noindent
In a similar manner, we can prove that the 2-dual space of $(L^1[0,1],\|.\|_1)$ is identified by $(Y_{[0,1] \times [0,1]}^\infty,\|.\|_{Y_{[0,1] \times [0,1]}^\infty})$
Now we shall discuss the 2-dual space of $\left (L^p[0,1],\|.,.\|_p^G \right ).  $ See \cite{Yos1}.
We need the concept of g-orthogonality on $L^p[0,1]$, where $g$ defined on
on $L^p[0,1] \times L^p[0,1]$ is given by the formula
$$ g(x,y):= \|x\|_p^{2-p} \int_0^1 |x(u)|^{p-1} \text{sgn}(x(u))y(u) du,\; x,y \in L^p[0,1].$$
(See \cite{Mili} and \cite{Yos1}.) 
Note that 
\begin{enumerate}
	\item $ g(x,x)=\|x\|_p^{2-p} \int_0^1 |x(u)|^{p-1} \text{sgn}(x(u))x(u) du=\|x\|_p^{2-p} \int_0^1 |x(u)|^{p-1} |x(u)| du$\\ As a consequence $g(x,x) =\|x\|_p^{2-p} \int_0^1 |x(u)|^p du=\|x\|_p^{2-p} \|x\|_p^p =\|x\|_p^2 .$ 
	\item $g(\alpha x, \beta y) = |\alpha|^{2-p} |\alpha|^{p-1}\text{sgn}\alpha \beta g(x,y)= \alpha \beta g(x,y)$
	\item $g(x,x+y)= \|x\|_p^2 +g(x,y) $.
	\item $|g(x,y)| \leq \|x\|_p^{2-p} \int_0^1 |x(u)|^{p-1} |y(u)| du=\|x\|_p^{2-p} \int_0^1 |x(u)|^\frac{p}{q} |y(u)| du \leq \|x\|_p^{2-p} \|x\|_p^\frac{p}{q} \|y\|_p $\\
	So  $|g(x,y)| \leq \|x\|_p \|y\|_p.$
	\item $g(x,y)$ is linear in $y.$
	
\end{enumerate}
As $g$ satisfies the above properties $g$ defines a semi inner product.
If $g(x,y)=0$, then we say that $x$ and $y$ are $g$  orthogonal and we write $x \perp_g y$ .
\noindent
Let $x \in L^p[0,1] $ and $Y=\left \lbrace y_1, y_2 \right \rbrace \subset L^p[0,1].$
Let $\Gamma(y_1,y_2)$ denote the Gram determinant $ \begin{vmatrix} g(y_1,y_1) & g(y_1,y_2) \\ g(y_2,y_1) & g(y_2,y_2) \end{vmatrix} $ 
If $\Gamma(y_1,y_2)\neq 0$ then the vector  $$x_Y:= -\frac{1}{\Gamma(y_1,y_2)}\begin{vmatrix}0 & y_1 & y_2 \\ g(y_1,x) & g(y_1,y_1) & g(y_1,y_2) \\ g(y_2,x) & g(y_2,y_1) & g(y_2,y_2) \end{vmatrix} $$ is the Gram Schmidth projection of vector $x$ on $Y.$ If $\left \lbrace x_1, x_2 \right \rbrace $ is linearly independent then 
$$ x_1^o= x_1,\;x_2^o= x_2 - \frac{g(x_1^o,x_2)}{g(x_1^o,x_1^o)} x_1^o $$ defines a left g-orthogonal sequence. Note that $x_1^o \perp_g x_2^o$. 
Define the volume of the 2-rectangle spanned by $x_1$ and $x_2$ by $$ V(x_1,x_2):= \|x_1^o\|\|x_2^o\|.$$ See \cite{Gun4}. If $x_1,x_2$ are linearly dependent define $V(x_1,x_2)=0.$ As in \cite{Gun4} it can be shown that
$$ V(x_1,x_2) \leq \|x_1,x_2\|_p^G$$ for all $x_1,x_2 \in L^p[0,1].$
Using the above result, the result  concerning the equivalence of the G\"{a}hler norm and Gunawan's norm can be exactly proved as \cite{Gun5}.
\begin{theorem}
	For $x_1, x_2 \in L^p[0,1]$ we have
	$$ 2^{\frac{1}{p}-1}\|x_1,x_2\|_p^H \leq \|x_1,x_2\|_p^G \leq 2^\frac{1}{p} \|x_1,x_2\|_p^H $$
	\end{theorem}\noindent 
Using the above result we have the following theorem which is similarly proved as in Theorem 2.3. of \cite{Yos1}

\begin{theorem}
	A bilinear 2-functional $f$ is bounded on $(L^p[0,1],\|.,.\|_p^G)$ if and only if $f$ is 
	antisymmetric and bounded on $(L^p[0,1],\|.\|_p)$. Furthermore, we have
	$$ \frac{1}{2}\|f\|_{2,1} \leq \|f\|_{2,2}^G \leq \|f\|_{2,1}$$ where
	$\|.\|_{2,2}^G$ is the norm on the 2-dual space of $(L^p[0,1],\|.\|_p^G). $
\end{theorem}

To identify the dual space of $(L^p[0,1],\|.\|_p^G)$ consider the subspace of $Y_{[0,1] \times [0,1]}^q $. Define a subspace $Z_{[0,1] \times [0,1]}^q $ to be all $\theta:[0,1]\times [0,1]\rightarrow \mathbb{R}$ measurable such that $\theta(u,v)=-\theta(v,u).$  $Z_{[0,1] \times [0,1]}^q $ can be viewed as a normed space equipped with norm inherited from $Y_{[0,1] \times [0,1]}^q $ We have shown that the 2-dual space of  $(L^p[0,1],\|.\|_p)$ is isometrically isomorphic to $(Y_{[0,1] \times [0,1]}^q,\|.\|_{Y_{[0,1] \times [0,1]}^q}).$ Hence the space of antisymmetric bounded bilinear 2-functionals on $(L^p[0,1],\|.\|_p)$ can be identified with $(Z_{[0,1] \times [0,1]}^q,\|.\|_{Y_{[0,1] \times [0,1]}^q})$. We now present the following Corollaries.

\begin{corollary}
	The function $\|.\|_{Y_{[0,1] \times [0,1]}^q}^G$ on $Z_{[0,1] \times [0,1]}^q$ defined by
$$ \|.\|_{Y_{[0,1] \times [0,1]}^q}^G:= \sup_{\|x,y\|_p^G \neq 0} \frac{|\int_0^1 \int_0^1 x(u) y(v) \theta(u,v) du dv|}{\|x,y\|_p^G}$$ defines a norm on $Z_{[0,1] \times [0,1]}^q$.
Further,  $\|.\|_{Y_{[0,1] \times [0,1]}^q}$ and  $\|.\|_{Z_{[0,1] \times [0,1]}^q}^G$ are equivalent norms on $Z_{[0,1] \times [0,1]}^q$ with
$$ \frac{1}{2}\|\theta\|_{Y_{[0,1] \times [0,1]}^q}\leq \|\theta\|_{Z_{[0,1] \times [0,1]}^q}^G \leq \|\theta\|_{Y_{[0,1] \times [0,1]}^q} $$ for all $ \theta \in Z_{[0,1] \times [0,1]}^q.$
\end{corollary}

\begin{corollary}
	The 2-dual space of  $(L^p[0,1],\|.,.\|_p^G)$ is isometrically isomorphic to  \\$(Z_{[0,1] \times [0,1]}^q,\|.\|_{Y_{[0,1] \times [0,1]}^q}^G)$.
\end{corollary}

\begin{corollary}
	The function $\|.\|_{Y_{[0,1] \times [0,1]}^q}^H$ on $Z_{[0,1] \times [0,1]}^q$ defined by
	$$ \|.\|_{Y_{[0,1] \times [0,1]}^q}^H:= \sup_{\|x,y\|_p^H \neq 0} \frac{|\int_0^1 \int_0^1 x(u) y(v) \theta(u,v) du dv|}{\|x,y\|_p^H}$$ defines a norm on $Z_{[0,1] \times [0,1]}^q$.
	Further,  $\|.\|_{Y_{[0,1] \times [0,1]}^q}^H$ and  $\|.\|_{Z_{[0,1] \times [0,1]}^q}^G$ are equivalent norms on $Z_{[0,1] \times [0,1]}^q$ with
	$$ 2^{\frac{1}{p}-1}\|\theta\|_{Y_{[0,1] \times [0,1]}^q}^G\leq \|\theta\|_{Z_{[0,1] \times [0,1]}^q}^H \leq 2^{\frac{1}{p}}\|\theta\|_{Y_{[0,1] \times [0,1]}^q}^G $$ for all $ \theta \in Z_{[0,1] \times [0,1]}^q.$
\end{corollary}

\begin{corollary}
	The 2-dual space of  $(L^p[0,1],\|.,.\|_p^H)$ is isometrically isomorphic to  \\$(Z_{[0,1] \times [0,1]}^q,\|.\|_{Y_{[0,1] \times [0,1]}^q}^H)$.
\end{corollary}

The identification of n-dual spaces for $L^p[0,1]$ can be done in exaxtly same way.
{\bf Acknowledgement}
The author would like to thank UGC FRP, INDIA  for their support.

\end{document}